\documentclass{gOPT2e}

\usepackage{amssymb,amsmath,amsfonts,float,subfig,graphics,epsfig}

\newcommand{\bff}{{\bf f}}
\newcommand{\bc}{{\bf c}}
\newcommand{\bx}{{\bf x}}
\newcommand{\by}{{\bf y}}

\newcommand{\bz}{{\bf z}}
\newcommand{\bA}{{\bf A}}
\newcommand{\bI}{{\bf I}}

\newcommand{\lam}{\lambda}
\newcommand{\Lam}{\Lambda}
\newcommand{\vsig}{\varsigma}

\newcommand{\barvsig}{\bar{\vsig}}
\newcommand{\barmu}{\bar{\mu}}
\newcommand{\barlam}{\bar{\lam}}
\newcommand{\barbz}{\bar{\bz}}
\newcommand{\barby}{\bar{\by}}
\newcommand{\barbx}{\bar{\bx}}

\newcommand{\calS}{{\cal S}}

\newcommand{\calSa}{\calS_a}
\newcommand{\calX}{{\cal X}}

\newcommand{\calXc}{{\calX_c}}
\newcommand{\calY}{{\cal Y}}

\newcommand{\calYc}{{\calY_c}}
\newcommand{\calZ}{{\cal Z}}

\newcommand{\calZc}{{\calZ_c}}

\newcommand{\calP}{{\cal P}}

\newcommand{\calV}{{\cal V}}
\newcommand{\calVa}{{\calV_a}}

\newcommand{\R}{I\!\!R}

\newcommand{\bgra}{{\bf \nabla}}

\newcommand{\eba}{\begin{array}}
\newcommand{\eea}{\end{array}}
\newcommand{\ebe}{\begin{eqnarray}}
\newcommand{\eee}{\end{eqnarray}}
\newcommand{\eb}{\begin{equation}}
\newcommand{\ee}{\end{equation}}

\begin{document}

\title{On the minimal distance between two surfaces}

\author{D.M. Morales Silva$^{\rm a}$$^{\ast}$\thanks{$^\ast$Corresponding %
author. Email: dmoralessilva@ballarat.edu.au \vspace{6pt}} and D. Y. %
Gao$^{\rm a}$\\\vspace{6pt} $^{\rm a}${\em{School of Science, Information Technology and Egineering, University of Ballarat, Victoria 3353}} }


\maketitle

\begin{abstract}
This article revisits previous results presented in \cite{DGAOY} which were challenged in \cite{V-Z}. We aim to use the points of view presented in \cite{V-Z} to modify the original results and highlight that the consideration of the so called Gao-Strang total complementary function is indeed quite useful for establishing necessary conditions for solving this problem.
\end{abstract}


\begin{classcode}
49K99; 49N15
\end{classcode}

\section{Introduction and Primal Problem}

Minimal distance problems between two surfaces arise naturally from many applications, which have been recently studied by both engineers and scientists (see \cite{JOHN-COHEN,PATO-GILLES}). In this article, the problem presents a quadratic minimization problem with equality constraints: we let $\bx:=(\by,\bz)$ and

\eb
(\calP):\ \min\left\{\Pi(\bx)=\frac{1}{2}\|\by-\bz\|^2:\ h(\by)=0,\ g(\bz)=0\right\},
\ee
where $h:\R^n\rightarrow \R$ and $g:\R^n\rightarrow \R$ are defined by
\eb\label{h}
h(\by):=\frac{1}{2}\left(\by^t\bA\by-r^2\right),
\ee
\eb\label{g}
 g(\bz):=\frac{1}{2}\alpha\left(\frac{1}{2}\|\bz-\bc\|^2-\eta\right)^2-\bff^t(\bz-\bc),
\ee
 in which,
 $\bA\in\R^{n\times n}$ is a positive definite matrix, $\alpha,r$ and $\eta$ are positive numbers,
  and $\bff,\bc\in \R^n$ are properly chosen so that these two surfaces
  \[
  \calYc:=\{\by\in\R^n:\ h(\by)=0\}
  \]
  and
  \[
  \calZc:=\{\bz\in\R^n:\ g(\bz)=0\}
   \]
   are disjoint such that  if $\bz\in \calZc$ then $h(\bz)> 0$.
   For example, it can be proved that if $\bc=0$, $r>0$, $\eta>0.5r^2$ and $\|\bff\|<0.5(0.5r^2-\eta)^2/r$ then, $\calYc\cap\calZc=\emptyset$ and if $\bz\in \calZc$ then $h(\bz)> 0$.  Notice that the feasible set $\calXc=\calYc\times\calZc\subset\R^n\times\R^n$, defined by $$\calXc=\{\bx\in\R^n\times\R^n:\ h(\by)=0,\ g(\bz)=0\},$$ is, in general, non-convex.

By introducing Lagrange multipliers $\lam,\mu\in\R$ to relax the two equality constraints in $\calXc$, the classical Lagrangian associated with the constrained problem $(\calP)$ is \eb \label{Lagrangian}
L(\bx,\lam,\mu)=\frac{1}{2}\|\by-\bz\|^2+\lam h(\by)+\mu g(\bz).
\ee

Due to the non-convexity of the constraint $g$, the problem may have multiple local minima.
The identification of the global minimizer has been a fundamentally difficult task in global optimization. The {\em canonical duality theory} is a  newly developed, potentially useful methodology, which is composed mainly of (i) a
\emph{canonical dual transformation}, (ii)  a {\em complementary-dual principle, }
and (iii) an associated \emph{triality theory}. The canonical dual transformation can be used to formulate dual problems
without duality gap; the complementary-dual principle shows that the canonical dual problem is equivalent to
the primal problem in the sense that they have the same set of KKT points;
 while the triality theory can be used to identify both global and local extrema.
 In global optimization, the canonical duality theory has been successfully used for solving many non-convex/non-smooth
 constrained optimization problems, including polynomial minimization \cite{DGAO1,DGAO4},
 concave minimization with inequality constraints \cite{DGAO3}, nonlinear dynamical systems  \cite{ruan-gao-ima},
 non-convex quadratic minimization with spherical \cite{DGAO2},  box \cite{DGAO5}, and integer constraints \cite{FANG}.

In the next section, we will show how to correctly 
use the canonical dual transformation to convert the non-convex constrained problem into a canonical dual problem. 
The global optimality condition is proposed in Section 2. Applications are illustrated in Section 3. 
The global minimizer is uniquely identified by the triality theory proposed in \cite{DGAOB}.

\section{Canonical dual problem}

In order to use the canonical dual transformation method, the key step is to introduce a so-called \emph{geometrical operator} $\xi=\Lambda(\bz)$ and a \emph{canonical function} $V(\xi)$ such that the non-convex function \eb\label{functionW}
W(\bz)=\frac{1}{2}\alpha\left(\frac{1}{2}\|\bz-\bc\|^2-\eta\right)^2
\ee in $g(\bz)$ can be written in the so-called canonical form $W(\bz)=V(\Lambda(\bz))$. 
By the definition introduced in \cite{DGAOB}, a differentiable function 
$V:\calVa\subset\R\rightarrow\calVa^*\subset\R$ is called a \emph{canonical function} if the duality relation 
$\vsig=DV(\xi):\calVa\rightarrow\calVa^*$ is invertible. Thus, for the non-convex function defined by 
\eqref{functionW}, we let $$\xi=\Lam(\bz)=\frac{1}{2}\|\bz-\bc\|^2,$$ then the quadratic function $V(\xi):=\frac{1}{2}\alpha(\xi-\eta)^2$ is a canonical function on the domain $\calVa=\{\xi\in\R:\ \xi\geq 0\}$
 since the duality relation 
 $$\vsig=DV(\xi)=\alpha(\xi-\eta):\calVa\rightarrow\calVa^*=\{\vsig\in\R:\ \vsig\geq -\alpha\eta\}$$ 
 is invertible. By the Legendre transformation, the conjugate function of $V(\xi)$ can be uniquely defined by
  \eb V^*(\vsig)=\{\xi\vsig-V(\xi):\ \vsig=DV(\xi)\}=\frac{1}{2\alpha}\vsig^2+\eta\vsig.
\ee

It is easy to prove that the following canonical relations  
\eb\label{canonicalrelations}
\xi=DV^*(\vsig)\Leftrightarrow\vsig=DV(\xi)\Leftrightarrow V(\xi)+V^*(\vsig)=\xi\vsig 
\ee 
hold in $\calVa\times\calVa^*$. 
Thus, replacing $W(\bz)$ in the non-convex function $g(\bz)$ by $V(\Lam(\bz))=\Lam(\bz)\vsig-V^*(\vsig)$, the non-convex Lagrangian $L(\bx,\lam,\mu)$ can be written in the Gao-Strang \emph{total complementary function} form \eb\label{XI}
\Xi(\bx,\lam,\mu,\vsig)=\frac{1}{2}\|\by-\bz\|^2+\lam h(\by)+\mu (\Lam(\bz)\vsig-V^*(\vsig)-\bff^t(\bz-\bc)).
\ee 
Through this total complementary function, the canonical dual function can be defined by 
\eb\label{XiPId} \Pi^d(\lam,\mu,\vsig)=\left\{\Xi(\bx,\lam,\mu,\vsig): \bgra_{\bx}\Xi(\bx,\lam,\mu,\vsig)=0\right\}.
\ee
Let  the  dual feasible space $\calSa$  be defined by 
\eb\label{setSa}
\calSa:=\{(\lam,\mu,\vsig)\in\R^3:(1+\mu\vsig)(\bI+\lam\bA)-\bI \text{ is invertible}\},
\ee 
where $\bI\in\R^{n\times n}$ is the identity matrix. 
Then the canonical dual function $\Pi^d$ is well defined by \eqref{XiPId}. 
In order to have the explicit form of $\Pi^d$, we need to calculate 
 \[
 \bgra_{\bx}\Xi(\bx,\lam,\mu,\vsig)=\left[\begin{array}{c}
\by-\bz+\lam\bA\by\\
\bz-\by+\mu\vsig(\bz-\bc)-\mu\bff
\end{array}\right].
\]
Clearly, if $(\lam,\mu,\vsig)\in\calSa$ we  have that 
$\bgra_{\bx}\Xi(\bx,\lam,\mu,\vsig)=0$ if and only if 
\eb\label{xbar}
\bx(\lam,\mu,\vsig)=\left[
\begin{array}{c}\mu((1+\mu\vsig)(\bI+\lam\bA)-\bI)^{-1}(\bff+\vsig\bc)\\
\mu(\bI+\lam\bA)((1+\mu\vsig)(\bI+\lam\bA)-\bI)^{-1}(\bff+\vsig\bc)
\end{array}\right].
\ee 
Therefore,
\[
\Pi^d(\lam,\mu,\vsig)=\Xi(\bx(\lam,\mu,\vsig),\lam,\mu,\vsig),
\]
 where $\bx(\lam,\mu,\vsig)$ 
is given by \eqref{xbar}.\\

The stationary points of the function $\Xi$  play a key role in  identifying the global minimizer of $(\calP)$.
 Because of this, let us put in evidence what conditions the stationary points of $\Xi$ must satisfy:

\ebe
\label{grad_x Xi}
& &\bgra_{\bx}\Xi(\bx,\lam,\mu,\vsig)=\left[\begin{array}{c}
\by-\bz+\lam\bA\by\\
\bz-\by+\mu\vsig(\bz-\bc)-\mu\bff
\end{array}\right]=0,\\
\label{grad_lam Xi}
& &\frac{\partial\Xi}{\partial\lam}(\bx,\lam,\mu,\vsig)=h(\by)=0,\\
\label{grad_mu Xi}
& &\frac{\partial\Xi}{\partial\mu}(\bx,\lam,\mu,\vsig)=\Lam(\bz)\vsig-V^*(\vsig)-\bff^t(\bz-\bc),\\
\label{grad_vsig Xi}
& &\frac{\partial\Xi}{\partial\vsig}(\bx,\lam,\mu,\vsig)=\mu(\Lam(\bz)-DV^*(\vsig)).
\eee

The following result can be found in \cite{V-Z}. Their proof will be presented for completeness.

\begin{lemma}\label{mu&lamnot0}
Consider $(\bx,\lam,\mu,\vsig)$ a stationary point of $\Xi$ then the following are equivalent:
\begin{enumerate}
\item[a)] $\mu=0$,
\item[b)] $\lam=0$,
\item[c)] $\bx\notin\calXc$.
\end{enumerate}
\end{lemma}
\begin{proof}
\begin{enumerate}
\item[a) $\rightarrow$ b)] If $\mu=0$, then from \eqref{grad_x Xi} we have $\by=\bz$. This implies that $\lam\bA\by=0$ but $\by\neq 0$ since $\|\by\|=r$ by \eqref{grad_lam Xi} and $\bA$ is invertible, therefore $\lam=0$.
\item[b) $\rightarrow$ c)] If $\lam=0$, then from \eqref{grad_x Xi}, $\by=\bz$ and so $(\by,\bz)\notin\calXc$ because $\calYc\cap\calZc=\emptyset$.
\item[c) $\rightarrow$ a)] Consider the counter-positive form of this statement, namely, if $\mu\neq 0$ then from \eqref{grad_vsig Xi}, $\Lam(\bz)=DV^*(\vsig)$ which combined together with \eqref{canonicalrelations} and \eqref{grad_mu Xi} provides $\bz\in\calZc$. Since $\by\in\calYc$, from \eqref{grad_lam Xi}, it has been proven that $\bx\in\calXc$.
\end{enumerate}
\end{proof}
Now we are ready to re-introduce Theorems 1 and 2 of Gao and Yang (\cite{DGAOY}).

\begin{theorem}\label{compl-dual princ}
(Complementary-dual principle) If $(\barbx,\barlam,\barmu,\barvsig)$ is a stationary point of $\Xi$ such that $(\barlam,\barmu,\barvsig)\in\calSa$ then $\barbx$ is a critical point of $(\calP)$ with $\barlam$ and $\barmu$ its Lagrange multipliers, $(\barlam,\barmu,\barvsig)$ is a stationary point of $\Pi^d$ and \eb \Pi(\barbx)=L(\barbx,\barlam,\barmu)=\Xi(\barbx,\barlam,\barmu,\barvsig)=\Pi^d(\barlam,\barmu,\barvsig).
\ee
\end{theorem}
\begin{proof}
From Lemma \ref{mu&lamnot0}, we must have that $\barlam$ and $\barmu$ are different than zero, otherwise they both will be zero and $(0,0,\vsig)\notin \calSa$ for any $\vsig\in\R$ which contradicts the assumption that $(\barlam,\barmu,\barvsig)\in\calSa$. Furthermore $\barbx\in\calXc$,
 clearly $\barbx$ is a critical point of $(\calP)$ with $\barlam$ and $\barmu$ its Lagrange multipliers and 
 $$\Pi(\barbx)=L(\barbx,\barlam,\barmu)=\Xi(\barbx,\barlam,\barmu,\barvsig).$$ On the other hand, since 
 $(\barlam,\barmu,\barvsig)\in\calSa$, Equations \eqref{xbar} and \eqref{grad_x Xi} are equivalent, therefore it is 
 easily proven that 
 $$\frac{\partial\Xi}{\partial t}(\barbx,\barlam,\barmu,\barvsig)=\frac{\partial\Pi^d}{\partial t}(\barlam,\barmu,\barvsig)=0,$$ 
 where $t$ is either $\lam,\mu$ or $\vsig$.
  This implies that $(\barlam,\barmu,\barvsig)$ is a stationary point of $\Pi^d$ and $$\Xi(\barbx,\barlam,\barmu,\barvsig)=\Pi^d(\barlam,\barmu,\barvsig)$$ 
  The proof is complete.
\end{proof}

Following the canonical duality theory, in order to identify the global minimizer of $(\calP)$, we first need to look at the Hessian of  $ \Xi$:
\eb\label{HessianXi_x}
\bgra^2_{\bx}\Xi(\bx,\lam,\mu,\vsig)=\left[\begin{array}{cc}
\bI+\lam\bA & -\bI \\
-\bI & (1+\mu\vsig)\bI
\end{array}\right].
\ee 
This matrix is positive definite if and only if $\bI+\lam\bA$ and $(1+\mu\vsig)(\bI+\lam\bA)-\bI$ 
are positive definite (see Theorem 7.7.6 in \cite{HORN}). 
With this, we define $\calSa^+\subset\calSa$ as follows: \eb\label{setSa+}
\calSa^+:=\{(\lam,\mu,\vsig)\in\calSa:\ \bI+\lam\bA \succ 0 \text{ and } (1+\mu\vsig)(\bI+\lam\bA)-\bI\succ 0 \}.
\ee

\begin{theorem}\label{necessaryconditions}
Suppose that $(\barlam,\barmu,\barvsig)\in\calSa^+$ is a stationary point of $\Pi^d$. Then $\barbx$ defined by \eqref{xbar} is the only global minimizer of $\Pi$ on $\calXc$.
\end{theorem}
\begin{proof}
Since $(\barlam,\barmu,\barvsig)\in\calSa^+$, it is clear that $\barbx\in\calXc$ and is the only global minimizer of 
$\Xi(\cdot,\barlam,\barmu,\barvsig)$. From \eqref{canonicalrelations}, notice that $V$ is a strictly convex function, 
therefore $V^*(\vsig)=\sup\{\xi\vsig-V(\xi):\xi\geq 0\}$ and 
\eb\label{Xi&L}
\Xi(\bx,\barlam,\barmu,\barvsig)\leq L(\bx,\barlam,\barmu),\ \forall \bx\in\R^{n\times n},
\ee 
in particular, $\Xi(\barbx,\barlam,\barmu,\barvsig)= L(\barbx,\barlam,\barmu)$. 
Suppose now that there exists $\bx'\in\calXc\setminus\{\barbx\}$ 
such that
 $$\Pi(\bx')\leq\Pi(\barbx),$$ 
 we would have the following: $$L(\bx',\barlam,\barmu)=\Pi(\bx')\leq\Pi(\barbx)=L(\barbx,\barlam,\barmu),$$
 but because of \eqref{Xi&L} this is equivalent to 
 $$\Xi(\bx',\barlam,\barmu,\barvsig)\leq L(\bx',\barlam,\barmu)\leq L(\barbx,\barlam,\barmu) = \Xi(\barbx,\barlam,\barmu,\barvsig).$$ 
 This contradicts the fact that $\barbx$ is the only global minimizer of $\Xi(\cdot,\barlam,\barmu,\barvsig)$, 
 therefore, we must have that $$\Pi(\barbx)<\Pi(\bx),\ \forall \bx\in\calXc\setminus\{\barbx\}.$$
\end{proof}

\begin{remark}
Notice that Theorem \ref{necessaryconditions} ensures that a stationary point in $\calSa^+$ will give us the only 
solution of $(\calP)$. Therefore, the existence and uniqueness of the solution of $(\calP)$ is necessary in order to 
find a stationary point of $\Pi^d$ in $\calSa^+$. From this it should be evident that the examples provided in 
\cite{V-Z} does not contradict any of the results established under the new conditions of Theorems 
\ref{compl-dual princ} and \ref{necessaryconditions}.
 It is a conjecture proposed in
\cite{DGAO5} that in nonconvex optimization with box/integer constraints, if the canonical dual problem
does not have a critical point in  $\calSa^+$, the primal problem could be NP-hard.  
\end{remark}

\section{Numerical Results}

The graphs in this section were obtained using WINPLOT \cite{PEANUT}.

\subsection{Distance between a sphere and a non-convex polynomial}

Let $n=3$, $\eta=2$, $\alpha=1,$ $\bff=(2,1,1)$, $\bc=(4,5,0)$, $r=2\sqrt{2}$ and $\bA=\bI$. In this case, the sets $\calSa$ and $\calSa^+$ are given by:
\eb\label{SaEx1}
\calSa=\{(\lam,\mu,\vsig)\in\R^3:(1+\mu\vsig)(1+\lam)\neq 1\},
\ee
\eb\label{Sa+Ex1}
\calSa^+=\{(\lam,\mu,\vsig)\in\R^3:1+\lam>0,\ (1+\mu\vsig)(1+\lam)> 1\}.
\ee

Using Maxima \cite{MAXIMA}, we can find the following stationary point of $\Pi^d$ in $\calSa^+$:

$$(\barlam,\barmu,\barvsig)=(0.9502828628898,1.06207786194864,0.30646555192966).$$ Then the global minimizer of $(\calP)$ is given by Equation \eqref{xbar}: $$\barby=\left(\begin{array}{c}
2.161477484004744\\
1.696777196962463\\
0.67004643869564
\end{array}\right),\ \
\barbz=\left(\begin{array}{c}
4.215492495576614\\
3.309195489378083\\
1.306780086728456
\end{array}\right).$$

\begin{figure}[h]
\centering
\includegraphics[scale=.5]{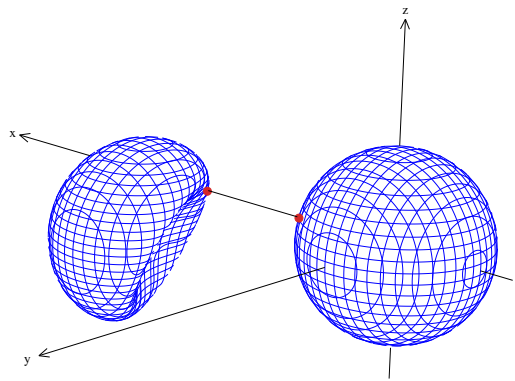}
\caption{Distance between a sphere and a non-convex polynomial}
\label{FigExample1}
\end{figure}

\subsection{Distance between an ellipsoid and a non-convex polynomial}

Let $n=3$, $\eta=2$, $\alpha=1,$ $\bff=(-2,-2,1)$, $\bc=(-4,-5,0)$, $r=2\sqrt{2}$ and $$\bA=\left[\begin{array}{ccc}
3 & 1 & 1 \\
1 & 4 & 1 \\
1 & 1 & 5
\end{array}\right].$$ Using Maxima \cite{MAXIMA}, we can find the following stationary point of $\Pi^d$ in $\calSa^+$: $$(\barlam,\barmu,\barvsig)=(0.84101802234162,1.493808342458642,0.12912817444352).$$ To put in evidence that this stationary point is in fact in $\calSa^+$, notice that the eigenvalues of $\bA$ are given by: \begin{eqnarray*}
\beta_1 = & \frac{4}{\sqrt{3}}\cos\left(\frac{4\pi}{3}+\frac{\theta}{3}\right)+4 & \approx 3.460811127 \\
\beta_2 = & \frac{4}{\sqrt{3}}\cos\left(\frac{2\pi}{3}+\frac{\theta}{3}\right)+4 & \approx 2.324869129 \\
\beta_3 = & \frac{4}{\sqrt{3}}\cos\left(\frac{\theta}{3}\right) +4 \verb|    | & \approx 6.214319743,
\end{eqnarray*} with $\displaystyle\theta=\cos^{-1}\left(\frac{3\sqrt{3}}{8}\right)$. Then, the matrices  $\bI+\barlam\bA$ and  $(1+\barmu\barvsig)(\bI+\barlam\bA)-\bI$ are similar to $$\left[\begin{array}{ccc}
3.910604529727413 & 0 & 0\\
0 & 2.955256837074665  & 0\\
0 & 0 & 6.226354900456345
\end{array}\right]$$ and $$\left[\begin{array}{ccc}
3.664931769065526 & 0 & 0\\
0 & 2.525304438283014  & 0\\
0 & 0 & 6.42737358375643
\end{array}\right]$$ respectively. Finally, the global minimizer of $(\calP)$ is given by Equation \eqref{xbar}: $$\barby=\left(\begin{array}{c}
-1.121270493506938\\
-0.83025443673537\\
0.66262025515374
\end{array}\right),\ \
\barbz=\left(\begin{array}{c}
-4.091279940255224\\
-4.009023330835817\\
1.807730500535487
\end{array}\right).$$

\begin{figure}[h]
\centering
\includegraphics[scale=.4]{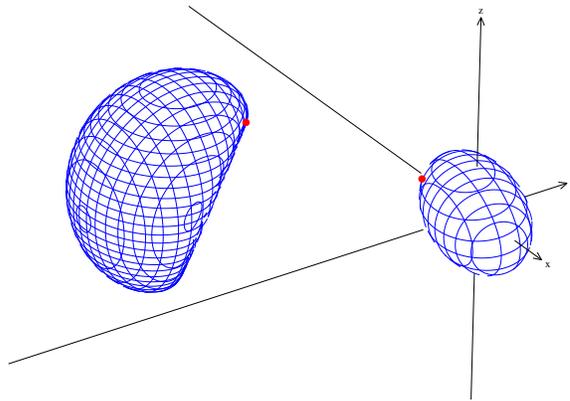}
\caption{Distance between an ellipsoid and a non-convex polynomial}
\label{FigExample2}
\end{figure}

\subsection{Example given in \cite{V-Z}}

Let $n=2$, $\alpha=\eta=1$, $\bc=(1,0)$, $\bff=\left(\frac{\sqrt{6}}{96},0\right)$, $r=1$ and $\bA=\bI$. As it was pointed out in \cite{V-Z}, there are no stationary points in $\calSa^+$. Under the new conditions of Theorem \ref{necessaryconditions}, this is expected since the problem has more than one solution (see figure \ref{FigExample3a}).
\begin{figure}[h]
\centering
\includegraphics[scale=.3]{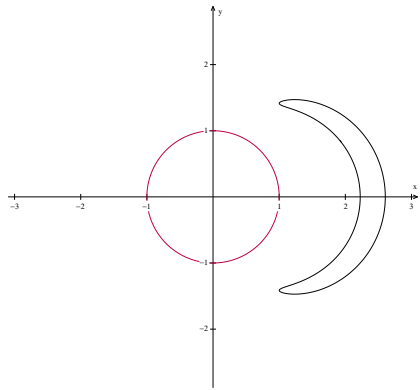}
\caption{Example given in \cite{V-Z}}
\label{FigExample3a}
\end{figure} The following was found (\cite{V-Z}) to be one of the global minimizers of $(\calP)$:
\[
\barby=\left(\begin{array}{c}
0.5872184947\\
0.8094284647
\end{array}\right),\ \
\barbz=\left(\begin{array}{c}
1.012757759\\
1.395996491
\end{array}\right).
\]
Notice that $\calSa$ and $\calSa^+$ are defined as in Equations \eqref{SaEx1} and \eqref{Sa+Ex1}.\\
In order to solve this problem, we will introduce a perturbation. Instead of the given $\bff$, we will consider $\bff_n=\left(\frac{\sqrt{6}}{96},\frac{1}{n}\right)$ for $n>100$.\\ The following table summarizes the results for different values of $n$.\\

\begin{tabular}{|c|c|c|}
\hline
 $n$ & $(\barlam_n,\barmu_n,\barvsig_n)\in\calSa^+$ & $\barbx_n=(\barby_n,\barbz_n)$\\ \hline
64 & (0.2284381,5.319007,-0.0219068) & $\barby=\left(\begin{array}{c}
0.2250312\\
0.9743515
\end{array}\right),\ \barbz=\left(\begin{array}{c}
0.2764370\\
1.1969306
\end{array}\right)$\\
\hline

1000 & (0.6926569,16.01863,-0.0248297) & $\barby=\left(\begin{array}{c}
0.5656039\\
0.8246770
\end{array}\right),\ \barbz=\left(\begin{array}{c}
0.9573734\\
1.3958953
\end{array}\right)$\\
\hline
10000 & (0.7214940,16.42599,-0.0254434) & $\barby=\left(\begin{array}{c}
0.5850814\\
0.8109745
\end{array}\right),\ \barbz=\left(\begin{array}{c}
1.0072142\\
1.3960878
\end{array}\right)$\\
\hline
100000 & (0.7243521,16.46345,-0.0255083) & $\barby=\left(\begin{array}{c}
0.5870050\\
0.8095833
\end{array}\right),\ \barbz=\left(\begin{array}{c}
1.0122034\\
1.3960066
\end{array}\right)$\\
\hline
\end{tabular}

\begin{figure}[h]
  \centering
  {\label{2DPert64}\includegraphics[width=0.3\textwidth]{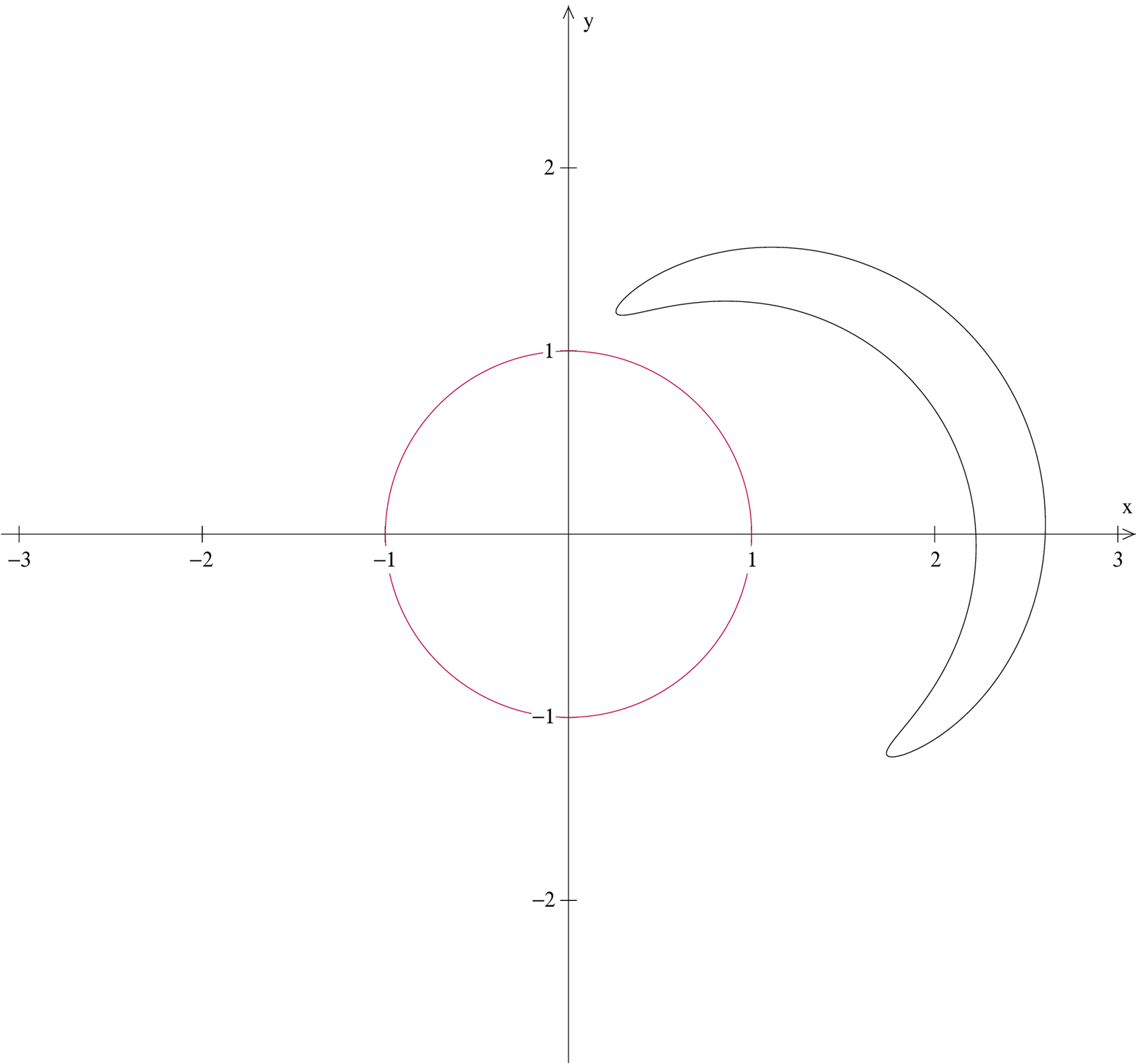}}\quad
 {\label{2DPert100K}\includegraphics[width=0.3\textwidth]{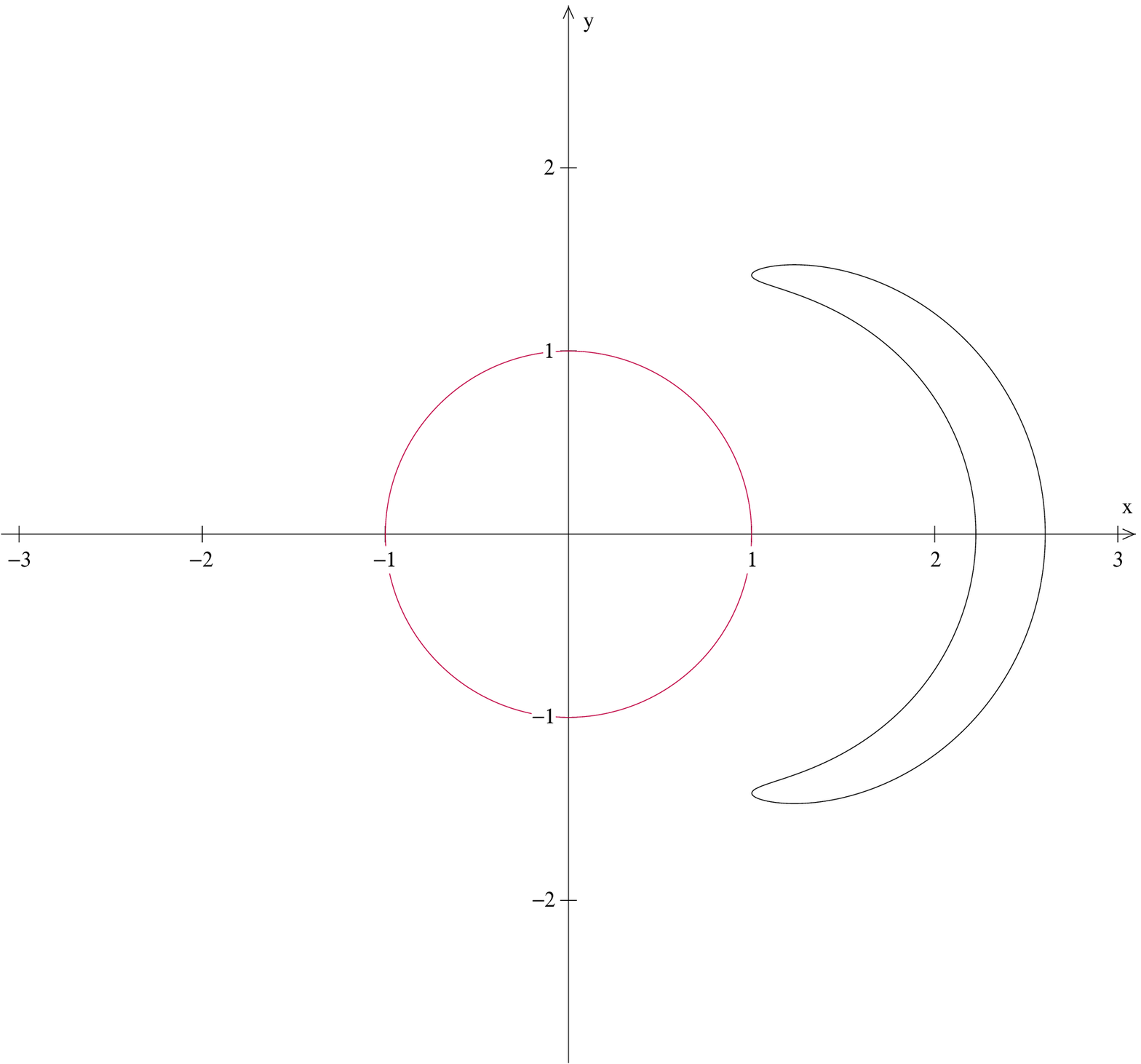}}
  \caption{Perturbations of Example given in \cite{V-Z}, $n=64$ to the left and $n=100000$ to the right.} \label{2DPert}
\end{figure}

\vspace{.3cm}

 \begin{remark}
The combination of the linear perturbation method and canonical duality theory
 for solving nonconvex optimization problems was first proposed in 
 \cite{r-g-j} with successful applications in solving some NP-complete problems \cite{wangetal}.
 High-order perturbation methods for solving integer programming problems were discussed in
 \cite{gao-ruan-jogo08}.
 \end{remark}

\section{Concluding remarks and future research}

\begin{enumerate}
\item[$\bullet$] The total complementary function (Equation \eqref{XI}) is indeed useful for finding necessary conditions for solving $(\calP)$ by means of the Canonical Duality theory.
\item[$\bullet$] The examples presented in \cite{V-Z} do not contradict the new conditions and results presented here.
\item[$\bullet$] As stated by Theorem \ref{necessaryconditions}, in order to use the canonical dual transformation a necessary condition is that $(\calP)$ has a unique solution. The question if this condition is sufficient remains open.
\item[$\bullet$] The combination of the perturbation and the canonical duality theory is an important method 
for solving nonconvex optimization problems which have more than one global optimal solution. 
    \item[$\bullet$] Finding a stationary point of $\Pi^d$ in $\calSa^+$ is not a simple task.
 It is worth to continue studying this problem in order to develop an  efficient algorithm for solving challenging
 problems in global optimization. 
 
\end{enumerate}

\end{document}